\documentclass[reqno,11pt]{amsart}
\usepackage{amssymb}
\usepackage[numbers]{natbib}
\usepackage[all]{xy}
\input epsf
\usepackage{lscape}

\numberwithin{figure}{section}
\newtheorem{thm}{Theorem}[section]
\newtheorem{lem}[thm]{Lemma}

\theoremstyle{definition}

\newtheorem{defn}[thm]{Definition}

\title{Distinguishing mutant pretzel knots in concordance}

\author{Allison N. Miller}
\address{Department of Mathematics, University of Texas, Austin, TX 78712, USA}

\begin{document}

\begin{abstract}
We prove that many pretzel knots of the form $K=P(2n,m,-2n \pm1,-m)$ are not topologically slice, even though their positive mutants $P(2n, -2n \pm1, m, -m)$ are ribbon. 
We use the sliceness obstruction of Kirk and Livingston \cite{KL99a} related to the twisted Alexander polynomials associated to prime power cyclic covers of knots. 

\end{abstract}
\bibliographystyle{alpha}

\maketitle

\section{Introduction}

A knot $K$ in $S^3$ is said to be {smoothly slice} if it bounds a smoothly embedded disc in $B^4$, and {topologically slice} if it bounds a locally flat embedded disc in $B^4$. A long-standing conjecture states that $K$ is smoothly slice if and only if it is ribbon \cite{Kirby78}-- i.e., if and only if $K$ bounds an immersed disc in $S^3$ with only ribbon self-intersections. 

Recently Lisca \cite{Lisca07} and Greene-Jabuka \cite{GJ13} proved this conjecture for the classes of 2-bridge knots and 3-strand pretzel knots with all parameters odd, respectively. Their arguments rely on sliceness obstructions associated to the double branched cover of the knot, and which come from Donaldson's intersection theorem and Heegaard Floer homology. 

\begin{thm}[\cite{Lisca07}]\label{rational}
The slice-ribbon conjecture holds for 2-bridge knots. 
\end{thm}

\begin{thm}[\cite{GJ13}]\label{3strand}
Let $P(p,q,r)$ be a 3-strand pretzel knot with $p, q, r$ odd and $|p|, |q|, |r| \geq 3$.\footnote{That is, such that $P(p,q,r)$ is not a 2-bridge knot.}
Then $P(p,q,r)$ is smoothly slice iff $p+q=0$, $p+r=0$, or $q+r=0$, and in each of these cases $P(p,q,r)$ is ribbon.  
\end{thm}

A natural extension of the case of 3-strand pretzel knots with all parameters odd is that of pretzel knots in general. Lecuona has results concerning the smooth slice status of pretzel knots with arbitrarily many strands and one even parameter \cite{Lec13}, while  Long has results on the sliceness of arbitrary 4- and 5-strand pretzels  \cite{Long14}. 
In particular, note that pretzel knots of the form $P(2n, -2n \pm1, m, -m)$ can easily be seen  to be ribbon. The following theorem, due independently to Lecuona and Long, establishes that up to reordering of the parameters these are in fact all of the smoothly slice 4-strand pretzel knots. 

\begin{thm}[\cite{Lec13}, \cite{Long14}] \label{4strand}
Suppose the pretzel knot $P(a,b,c,d)$ is smoothly slice. 
Then $\{a,b,c,d\}=\{2n, -2n \pm1, m, -m\}$ for some $m, n \in \mathbb{Z}$. 
\end{thm}

In particular, the only 4-strand pretzel knots whose smooth slice status is still unresolved are the knots $P(2n, m, -2n \pm 1, -m)$ that are positive mutants of the ribbon knots $P(2n, -2n \pm 1, m, -m)$.  However, the arguments used by Lisca, Greene-Jabuka, Lecuona, and Long in the proofs of the above theorems all rely on smooth sliceness obstructions that are associated to the double branched cover of a knot, and so which automatically vanish on mutants of smoothly slice knots. 

The twisted Alexander polynomials associated to cyclic covers of knots are powerful tools for distinguishing knots from their mutants, even up to topological concordance, 
as demonstrated by Livingston et al in \cite{KL99b}, \cite{HKL10}, and \cite{Liv09}. For example, Herald, Kirk, and Livingston demonstrate in \cite{HKL10}  that the 24 distinct oriented mutants of $P(3,7,9,11,15)$ are mutually distinct in the topological concordance group.

We use twisted Alexander polynomials to show that many 4-strand pretzel knots of the form $P(2n,m,-2n\pm1,-m)$ are not even topologically slice, though their positive mutants $P(2n,-2n \pm 1,m,-m)$ are ribbon.  Note that by considering $-K$ we can assume without loss of generality that $n>0$. 

\begin{thm}\label{maintheorem}
Suppose $n \in \mathbb{N}$ and $m \in \mathbb{Z}$ are such that $m$ is odd and there exists a prime $p$ dividing $m$ such that  
\begin{itemize}
\item 2 is a primitive root mod $p$.
\item $p$ does not divide $2n(2n \pm 1) $
\item $n \geq \frac{p+1}{2}$. 
\end{itemize}
 Also, assume that $(n,p) \neq (3,5)$. 
Then $K_{m,n}^{\pm}= P(2n, m, -(2n\pm 1), -m)$ is not topologically slice.
\end{thm}

The argument proceeds very similarly in the two cases of $K_{m,n}^+=P(2n, m, -2n-1, -m)$ and $K_{m,n}^{-}=P(2n,m,-2n+1,-m)$. In the following, we focus on the first case $K_{m,n}:=K_{m,n}^{+}$, leaving the precise statement and verification of the corresponding results for $K_{m,n}^-$  almost entirely to the reader. 

\section{Background}

In general, twisted homology can be defined for spaces $X$ which are homotopy equivalent to finite CW complexes as follows. (See \cite{KL99a} and \cite{HKL10} for a more thorough exposition.) 

Let $\tilde{X}$ denote the universal cover of $X$, so $C_*(\tilde{X})$ is acted on by the left by $\pi=\pi_1(X)$. Given $M$ a $(S, \mathbb{Z}[\pi])$ bimodule, the \emph{twisted chain complex} is defined as  $C_*(X, M)= C_*(\tilde{X}) \otimes_{\mathbb{Z}[\pi]} M$. The twisted chain complex $C_*(X,M)$ inherits a left $S$-module structure from $M$, which descends to the \emph{twisted homology} $H_k(X,M)= H_k(C_*(X,M))$. 
In particular, when $S=\mathbb{F}[t^{\pm1}]$ the k$^{th}$ twisted Alexander polynomials of $X$ are defined as follows.

\begin{defn} 
Let $M$ be a $(\mathbb{F}[t^{\pm1}], \mathbb{Z}[\pi])$-bimodule. The $k^{th}$ \emph{twisted Alexander polynomial} $\Delta^k_{X,M}(t)$ associated to $X$ and $M$ is the order of $H_k(X,M)$ as a $\mathbb{F}[t^{\pm1}]$ module. When $k=1$, we often call $\Delta^1_{X,M}(t)=: \Delta_{X,M}(t)$ \emph{the twisted Alexander polynomial}. 
\end{defn}

Note that twisted Alexander polynomials are only defined up to multiplication by units, which for $\mathbb{F}[t^{\pm1}]$ are of the form $\lambda t^j$ for $\lambda \in \mathbb{F}$ and $j \in \mathbb{Z}$.

In particular, we will be interested in the twisted Alexander polynomials of prime power cyclic covers of knot exteriors, with $M= \mathbb{F}[t^{\pm1}] \otimes_\mathbb{F} V$ for $V$ a finite dimensional $\mathbb{F}$-vector space. We will now define some notation (again following that of \cite{HKL10}) to be used throughout:

\begin{enumerate}

\item Given  $V$ a finite dimensional vector space over a field $\mathbb{F}$ and maps $\epsilon: \pi_1(X) \to \mathbb{Z}$ and $\phi: \pi_1(X) \to GL(V)$, then $M= \mathbb{F}[t^{\pm1}] \otimes_\mathbb{F} V$ has the natural left $\mathbb{F}[t^{\pm1}]$-module structure and has right $\mathbb{Z}[\pi_1(X)]$-module structure given by $\epsilon \otimes \phi$; that is,  
\[ (p(t), v) \cdot \gamma= (t^{\epsilon(\gamma)} p(t), v\phi(\gamma)), \text{ for } \gamma \in \pi_1(X)\]
We will often call the corresponding twisted Alexander polynomial  $\Delta_{X, \epsilon \otimes \phi}(t)$. 

\item Given $X, \epsilon, \phi$ as above, the \emph{reduced Alexander polynomial} is  $ \widetilde{\Delta}_{X, \epsilon \otimes \phi}(t)= \Delta_{X, \epsilon \otimes \phi}(t) (t-1)^{-s}, \text{ where } s=0 \text{ if } \phi \text{ is trivial, } s=1 \text{ else. }$

\item For $K$ a knot, let $X(K):=S^3- \nu(K)$ denote the exterior of $K$,  $X_n(K)$ denote the $n$-fold cyclic cover of $X(K)$, and $\Sigma_n(K)$ denote the corresponding $n$-fold branched cover of $S^3$ along $K$.  Finally, in contexts where $K$ is clear, let $\pi=\pi_1(X(K))$ and $\pi_n= \pi_1(X_n(K))$. 

\item Let $\epsilon: \pi_1(X(K)) \to H_1(X(K)) \cong \mathbb{Z}$ be the Hurewicz abelianization map. Note that $\epsilon$  maps $\pi_1(X_n(K)) \subset \pi_1(X(K))$ onto $n\mathbb{Z} \subset \mathbb{Z}$, so we can define
 $\epsilon_n: \pi_n \twoheadrightarrow \mathbb{Z}$ as the composition
$\epsilon_n: \pi_n \hookrightarrow \pi \xrightarrow{\epsilon} n\mathbb{Z} \twoheadrightarrow \mathbb{Z}$.

\end{enumerate}

\begin{defn}
Let $\mathbb{F} \subseteq \mathbb{C}$ . Define an involution of $\mathbb{F}[t^{\pm1}]$ by
\[\bar{ } : \mathbb{F}[t^{\pm1}]  \to  \mathbb{F}[t^{\pm1}] \text{, } f(t)=\sum_{j=m}^{n} a_j t^j  \mapsto \sum_{j=m}^n \overline{a_j} t^{-j}=\overline{f(t)}\]
A polynomial $g(t) \in \mathbb{F}[t^{\pm1}]$ is a \emph{norm in }$\mathbb{F}[t^{\pm1}]$ if $g(t)=\lambda t^k f(t)\overline{f(t)}$ for some $\lambda \in \mathbb{F}$, $k \in \mathbb{Z}$, and $f(t) \in \mathbb{F}[t^{\pm1}]$. 

\end{defn}

We will now state the major obstruction to sliceness coming from twisted Alexander polynomials. 
First, observe that given any $\chi: H_1(X_n) \to \mathbb{Z}_{m}$ and $\xi_m$ a primitive $m^{th}$ root of unity, there is
$ \phi_\chi: \pi_n \xrightarrow{ab} H_1(X_n)\to\mathbb{Q}(\xi_m)^{\times}=GL(\mathbb{Q}(\xi_m))$ given by $\phi_\chi(\gamma)= \xi_m^{\chi(\gamma)}$. Note that, here and otherwise, we will abuse notation by using $\gamma$ to refer to both an element of $\pi_1(X_n)$ and its image in $H_1(X_n)$.

In \cite{KL99a}, the following theorem is proved by establishing a relationship between twisted Alexander polynomials of $X_n(K)$ and corresponding twisted Reidemeister torsions of $\Sigma_n(K)$, and then using duality results for Reidemeister torsion.

\begin{thm}[\cite{KL99a}]\label{notsliceqodd}
Let $K$ be a topologically slice knot and $p, q$ be distinct primes, $q \neq 2$. Let $m=p^r$ and $d=q^s$ be prime powers.
 Then there exists an invariant metabolizer $M< H_1(\Sigma_m(K))$ such that for any $\chi: H_1(X_m(K)) \to \mathbb{Z}_d$ that factors through $H_1(\Sigma_m(K))$ and vanishes on $M$, 
  the corresponding reduced twisted Alexander polynomial 
  $\widetilde{\Delta}_{X_m, \epsilon_m \otimes \phi_\chi}(t) \in \mathbb{Q(}\xi_d)[t^{\pm1}]$
factors as a norm in $\mathbb{Q}(\xi_d)[t^{\pm1}]$. 
\end{thm}

However, as observed by Long in \cite{Long14}, the pretzel knots $K_{m,n}^{\pm}$ have only 2-torsion in their prime power cyclic branched covers. 
So we will need the following theorem, which follows immediately from the proof of Theorem \ref{notsliceqodd}, as observed by \cite{Liv09}. 

\begin{thm}[\cite{KL99a}]\label{notsliceq2}
Let $K$ be a topologically slice knot,  $p \neq 2$ prime, $m=p^r$ and $d=2^s$.  
Then there exists an invariant metabolizer $M < H_1(\Sigma_m(K))$ such that for any $\chi: H_1(X_m(K)) \to \mathbb{Z}_d$ that factors through $H_1(\Sigma_m(K))$ and vanishes on $M$, the corresponding reduced twisted Alexander polynomial $\widetilde{\Delta}_{X_m, \epsilon_m \otimes \phi_\chi}(t) \in \mathbb{Q}(\xi_d)[t^{\pm1}]$ factors as a norm in some $\mathbb{Q}(\xi_{2^n})[t^{\pm1}]$.  
\end{thm}

Note that the difference between the two theorems comes in whether we can assume that the reduced twisted Alexander polynomial factors as a norm over the field $\mathbb{Q}(\xi_d)$ that its coefficients naturally lie in (as in Theorem \ref{notsliceqodd}) or only in some larger cyclotomic extension (Theorem \ref{notsliceq2}).  We will be interested in Theorem \ref{notsliceq2} in the case $d=2$, when the reduced twisted Alexander polynomial will lie in $\mathbb{Q}[t^{\pm1}]$ and we will need to obstruct its factoring as a norm in $\mathbb{Q}(\xi_{2^n})[t^{\pm1}]$ for any $n \in \mathbb{N}$. In fact, we will show that the resulting reduced polynomials do not even factor as norms in $\mathbb{C}[t^{\pm1}]$, relying heavily on the fact that all coefficients are real. 

\medskip

In our application of Theorem \ref{notsliceq2}, we will rely on the observation of \cite{HKL10} that when $H=H_1(\Sigma_m(K), \mathbb{Z}_d)$ is irreducible as a $\mathbb{F}_d[\mathbb{Z}_m]$-module, any invariant metabolizer $M< H_1(\Sigma_m (K))$ must have trivial image $\overline M < H$. So any $\chi: H_1(X_m(K)) \to H_1(\Sigma_m(K))\to \mathbb{Z}_d$ must vanish on $M$, and if $K$ is slice then the reduced twisted Alexander polynomial associated to such a $\chi$ must factor as a norm. Therefore, when  $H$ is irreducible the computation of a single twisted Alexander polynomial can obstruct $K$'s sliceness. 
However, when $H$ is not irreducible a more involved decomposition of $H$ into irreducible components,  analysis of potential metabolizers,  construction of characters vanishing on said metabolizers, and computation of the corresponding twisted Alexander polynomials is required.\footnote{Example computations suggest that in the cases of interest the relevant twisted Alexander polynomials also increase significantly in complexity.}

In this context, our requirements that 2 is a primitive root mod $p$, that $p$ divides $m$, and that $p$ does not divide $2n(2n+1)$ are exactly those that establish that $H_1(\Sigma_p(K_{m,n}), \mathbb{Z}_2)$ is a nontrivial irreducible $\mathbb{F}_2[\mathbb{Z}_p]$-module (Lemma \ref{homcomp}) and hence exactly those that
allow us to obstruct the sliceness of $K_{m,n}$ by computing a single twisted Alexander polynomial. Note that our requirement that $n \geq \frac{p+1}{2}$ is not relevant to irreducibility; however, when $n< \frac{p+1}{2}$, the twisted Alexander polynomials we compute are norms even in $\mathbb{Q}[t^{\pm1}]$. 

\medskip

Before proving our result, we will need some computational results from \cite{Wada94} and \cite{HKL10}. 

\subsection{Computing with Fox derivatives}
First, Wada\footnote{Note that Wada's definition of a twisted Alexander polynomial differs from the one given above-- an equivalence is proven in \cite{KL99a}.} provides a way to compute twisted Alexander polynomials via Fox derivatives.

Suppose that $\pi_1(X)= \langle x_1, \dots, x_s : r_1, \dots, r_t \rangle$. Let $\rho: \pi_1(X) \to GL_n(\mathbb{F})$ and $\epsilon: \pi_1(X) \to \mathbb{Z}$ be nontrivial. Let 
$\Phi$ be the composition 
\[\Phi: \mathbb{Z}[\langle x_1, \dots, x_s \rangle] \twoheadrightarrow \mathbb{Z}[\pi] \xrightarrow{\epsilon \otimes \rho} M_n(\mathbb{F}[\mathbb{Z}]).\]

Then the  twisted homology $H_*(\pi, \mathbb{F}[\mathbb{Z}]^n)$ can be computed via chain complex
\[ \dots \to \left( \mathbb{F}[\mathbb{Z}]^n \right)^t \xrightarrow{\delta_2}  \left( \mathbb{F}[\mathbb{Z}]^n \right)^s
\xrightarrow{\delta_1}  \mathbb{F}[\mathbb{Z}]^n \to 0 \]
\begin{align*} \text{ where }
\delta_2= \left[\Phi \left( \frac{ \partial r_i}{\partial x_j} \right) \right] _{tn,sn}
 \text{ and } 
 \delta_1= \left[ \begin{array}{c} \Phi(x_1-1) \\ \vdots \\ \Phi(x_s-1) \end{array} \right].
\end{align*}

\begin{thm}[\cite{Wada94}, \cite{KL99a}] \label{foxder}
With the setup above, there is some $j$ such that $\Phi(x_j-1)$ has nonzero determinant. 
Let $p_j: (\mathbb{F}[\mathbb{Z}]^n)^s \to  (\mathbb{F}[\mathbb{Z}]^n)^{s-1}$ be the projection with kernel the $j$-th copy of $\mathbb{F}[\mathbb{Z}]^n$.  
Define $Q_j \in \mathbb{F}[\mathbb{Z}]$ to be the greatest common divisor of the 
$n(s-1) \times n(s-1)$ subdeterminants of the matrix for $p_j \circ \delta_2: (\mathbb{F}[\mathbb{Z}]^n)^t \to  (\mathbb{F}[\mathbb{Z}]^n)^{s-1}$. 
Then, when $H_1(X, \mathbb{F}^n[\mathbb{Z}])$ is torsion,
\[ \Delta_1 (X) = Q_j \frac{ \Delta_0(X)}{\det(\Phi(x_j-1))}\]
\end{thm} 

In our case, we will have a generator $x_j$ in $\pi_1(K)$ with $\chi(x_j)=0$ and $\epsilon(x_j)=1$,  so $\Delta_0(X)=1$. In addition, we will choose $\rho$ so that for some generator $x_j$, we have $\det(\Phi(x_j-1))=1-t$. Finally, we will work with a reduced Wirtinger presentation, which has deficiency one and hence eliminates the need to take greatest common divisors.  So we will have $\Delta_1(X)= \det \Phi(Z) (1-t)^{-1}$, where $Z$ is obtained from $\left[ \frac{\partial r_i}{\partial x_j} \right]_{s-1,s} $ by deleting the column corresponding to $x_j$.

\subsection{Covers and Shapiro's lemma}

We will also use the following theorem of \cite{HKL10} that relates certain twisted Alexander polynomials of covers to those of the base space. 

Let $p,q$ be distinct primes.
Recall that we have 
\begin{itemize}
\item $X=X(K)$ with $\pi=\pi_1(X)$.
\item A canonical $\epsilon: \pi \to \mathbb{Z}= \langle x \rangle$ inducing $p$-fold cyclic cover $X_p \to X$ and corresponding surjection $\epsilon_p: \pi_1(X_p) \to \mathbb{Z}$. 
\item A choice of meridian $\mu \in \pi$ with $\epsilon(\mu)=1$. 
\end{itemize}
Now, suppose that we have an irreducible $\mathbb{F}_q[\mathbb{Z}_p]$-module $V$, 
 a nonzero equivariant\footnote{i.e. $\rho(\mu \gamma \mu^{-1}) = t \cdot \rho(\gamma)$ for any $\gamma \in \pi_1(X_p)$ and $\mu$ our preferred meridian.}
 homomorphism $\rho: \pi_1(X_p) \to V$,
and 
a $\mathbb{Z}_q$-vector space homomorphism $\chi: V \to \mathbb{Z}_q$.
We would like to compute the twisted Alexander polynomial $\Delta_{X_p, \epsilon_p \otimes \rho_\chi}(t)$.

First, note that
there is a group structure on $\mathbb{Z} \ltimes V$ given by $(x^i,v) \cdot (x^j, w):= (x^{i+j}, t^{-j}\cdot v+w)$), where the action of $t$  on $V$ is given by $V$'s structure as a $\mathbb{F}_q[\mathbb{Z}_p]$-module. Since $\rho$ is equivariant, there is a well-defined extension of $\epsilon|_{\pi_p} \times \rho: \pi_p \to p\mathbb{Z} \times V$ to a homomorphism $\tilde{\rho}: \pi \to \mathbb{Z} \ltimes V$ defined by 
$ \tilde{\rho}(\gamma)= (x^{\epsilon(\gamma)}, \rho(\mu^{-\epsilon(\gamma)}\gamma)).$
(In fact,  \cite{HKL10} shows that this defines a bijection between equivariant $\rho$ and homomorphisms $\tilde{\rho}$ with $\tilde{\rho}(\mu)=(x,0)$.)

Now, define a map $\Phi: \pi_1(X) \to GL_p(\mathbb{Q}(\xi_q)[t^{\pm1}])$ as the composition of $\tilde{\rho}$ with the following map $\mathbb{Z} \ltimes V \to GL_p(\mathbb{Q}(\xi_q)[t^{\pm1}])$:
\[
(x^j, v) \mapsto
 \left[
 \begin{array}{cccc}
 0& 1&\cdots & 0 \\
 \vdots & \vdots & \ddots & \vdots\\
 0 & 0 & \cdots & 1\\
 t & 0 & \cdots &0
 \end{array}
 \right]^j
 \left[
 \begin{array}{cccc}
 \xi_q^{\chi(v)} & 0 & \cdots & 0\\
 0 & \xi_q^{\chi(t \cdot v)}& \cdots &0 \\
 \vdots & \vdots & \ddots& \vdots \\
 0 & 0 & \cdots & \xi_q^{\chi(t^{p-1} \cdot v)} 
 \end{array}
 \right]
\]

\begin{thm}[\cite{HKL10}]\label{covercomp}
Let $X, X_p, \epsilon, \rho,$ and $\Phi$ be as above, where
\begin{itemize}
\item
 $\epsilon \otimes \rho_\chi: \pi_1(X_p) \to GL_1(\mathbb{Q}(\xi_q)[t^{\pm1}])$ gives $\mathbb{Q}(\xi_q)[t^{\pm1}]$ a $(\mathbb{Q}(\xi_q)[t^{\pm1}], \mathbb{Z}[\pi_1(X_p)])$- bimodule structure
\item $\Phi: \pi_1(X) \to GL_p(\mathbb{Q}(\xi_q)[t^{\pm1}])$ gives $(\mathbb{Q}(\xi_q)[t^{\pm1}])^p$ a $(\mathbb{Q}(\xi_q)[t^{\pm1}], \mathbb{Z}[\pi_1(X)])$- bimodule structure.
\end{itemize}
Then the corresponding twisted homology groups $H_1(X_p,\mathbb{Q}(\xi_q)[t^{\pm1}])$ and $H_1(X, (\mathbb{Q}(\xi_q)[t^{\pm1}])^p)$ are isomorphic as $\mathbb{Q}(\xi_q)[t^{\pm1}]$-modules, and so
$ \Delta_{X_p, \epsilon \otimes \rho_\chi}(t) = \Delta_{X, \Phi}(t)$ as well. 
\end{thm}

This result, when combined with the Fox derivative computational result Theorem \ref{foxder} will allow us to compute twisted Alexander polynomials for $X_p$ directly from a Wirtinger presentation for $\pi_1(X)$.This will simplify the computation, even though the representations increase correspondingly in complexity, mapping elements of the fundamental group to $p \times p$ instead of $1 \times 1$ matrices.


\section{Main theorem}

Our main Theorem \ref{maintheorem} will follow almost immediately from a series of lemmas and computations. 

\begin{proof}[Proof of Theorem \ref{maintheorem}]
First, note that since $P(2n,m,-2n-1,-m)$ and $P(2n,-m,-2n-1,m)$ are the same as unoriented knots, we can assume without loss of generality that $m>0$. 
By Lemma \ref{homcomp} we have that $H_1(\Sigma_p(K), \mathbb{Z}_2)$ is irreducible. Therefore, as observed by \cite{HKL10}, any metabolizer $M \leq H_1(\Sigma_p(K))$ must have trivial image in $H_1(\Sigma_p(K), \mathbb{Z}_2)$. 
So any map $ H_1(X_p(K)) \to \mathbb{Z}_2$ that factors through $H_1(\Sigma_p(K))$ vanishes on $M$. Therefore, to obstruct $K$'s sliceness it suffices to show that there is some such map such that the corresponding reduced twisted Alexander polynomial is not a norm in $\mathbb{C}[t^{\pm1}]$, and hence not in any $\mathbb{Q}(\xi_{2^k})[t^{\pm1}]$. In the following, we construct this map, compute the corresponding reduced twisted Alexander polynomial explicitly (Lemma \ref{comp}), and show that this polynomial is not a norm in $\mathbb{C}[t^{\pm1}]$ (Lemma \ref{notnorm}), except when $n=3$ and $p=5$.  
\end{proof}

We will often write $m=2k+1$. Note that we can also write $m= p + 2jp$ for some $j \in \mathbb{N}$, so $k= \frac{m-1}{2}= jp + \frac{p-1}{2}$. 

\subsection{Homology computation for the branched covers.}

\begin{lem}\label{homcomp}
Let $p,m,n \in \mathbb{N}$ be as above. 
Then $H_1(\Sigma_p(K_{m,n}), \mathbb{Z}_2)$ is isomorphic to the irreducible $\mathbb{F}_2[\mathbb{Z}_p]$-module $V_p=\mathbb{F}_2[t] / \sum_{i=0}^{p-1} t^i$.
\end{lem}

\begin{proof}
First, observe that there is a  Seifert matrix for $K_{m,n}$ given by $A_{m,n}$ as follows:

\begin{align*}
A_{m,n}=  \left[ \begin{array}{ccccc}
-B_{2n-1} & 0 & 0 & 0 &0 \\
0 & - B_{m-1}^T &0&0  & -U_{m-1}^T \\
0 & 0 & B_{2n}^T &0& U_{2n}^T \\
0 & 0 &0 & B_{m-1} & 0 \\
-U_{2n-1} & 0 & 0 & U_{m-1} & 0
\end{array}\right], \text{ where}
\end{align*}

\begin{align*}
B_k= 
\left[
\begin{array}{cccccc}
1 & -1 & 0 & 0& \cdots & 0 \\
0& 1 &-1 & 0 & \cdots &0\\
0 & 0 & 1 & -1 & \cdots & 0 \\
\cdots  & \cdots & \cdots & \cdots & \cdots & \cdots \\
0 & 0 & \cdots & 0 & 1 & -1 \\
0& 0 & \cdots & 0 & 0 & 1
\end{array}
\right]_{k,k}
\text{ and }
U_k= 
\left[
\begin{array}{ccccc}
1  & 0 & 0 & \cdots & 0 
\end{array} 
\right]_{1,k}. 
\end{align*}

Note that by taking the determinant of  $tA_{m,n}- A_{m,n}^T$ we can observe that the Alexander polynomial of $K_{m,n}$ is 
 \[\Delta_{m,n}(t)= \left( \sum_{i=0}^{m-1} (-t)^i \right)^2.\]
 Now reduce $tA_{m,n}-A_{m,n}^T$ over $\mathbb{Z}_2$ coefficients to get a new presentation of $H_1(X_{\infty}(K_{m,n}), \mathbb{Z}_2)$ as a $\mathbb{F}_2[\mathbb{Z}]$-module:
\[ \left[ \begin{array}{cc}
    \sum_{i=0}^{m-1} t^i &\left(  \sum_{i=0}^{2n} t^i \right) \left( \sum_{i=0}^{2n-1} t^i \right)  \\ 
    0 & \sum_{i=0}^{m-1} t^i \\ 
  \end{array}\right].
\]

Note that $H_1( \Sigma_p(K_{m,n}), \mathbb{Z}_2)$ is naturally a $\mathbb{F}_2[\mathbb{Z}_p]$-module, with the $\mathbb{Z}_p$ action coming from the covering transformation. In addition, this module is obtained by imposing the relation $\sum_{i=0}^{p-1} t^i = \frac{t^p-1}{t-1}$ on $H_1(X_{\infty}, \mathbb{Z}_2)$. 
So $H_1( \Sigma_p(K_{m,n}), \mathbb{Z}_2)$ is presented by 
\[ \left[ \begin{array}{cc}
    \sum_{i=0}^{m-1} t^i &\left(  \sum_{i=0}^{2n} t^i \right) \left( \sum_{i=0}^{2n-1} t^i \right)  \\ 
    0 & \sum_{i=0}^{m-1} t^i \\ 
 \sum_{i=0}^{p-1} t^i & 0 \\
 0 &  \sum_{i=0}^{p-1} t^i \\
  \end{array}\right] \approx
\left[ \begin{array}{cc}
    \sum_{i=0}^{p-1} t^i &\left(  \sum_{i=0}^{2n} t^i \right) \left( \sum_{i=0}^{2n-1} t^i \right)  \\ 
    0 & \sum_{i=0}^{p-1} t^i \\ 
  \end{array}\right]
\]

It is a well known fact that since 2 is a primitive root mod $p$, the polynomial $\sum_{i=0}^{p-1} t^i$ is irreducible in $\mathbb{F}_2[t]$.  Note that $p$ does not divide $2n+1$ or $2n$, and so $\sum_{i=0}^{p-1} t^i$ does not divide $ \sum_{i=0}^{2n} t^i$ or $ \sum_{i=0}^{2n-1} t^i$ and therefore is relatively prime to both of them in $\mathbb{F}_2[\mathbb{Z}]$. 

Therefore, we can apply the Euclidean algorithm in  $\mathbb{F}_2[\mathbb{Z}]$ via Tietze-like moves to simplify the above matrix and demonstrate that $H_1( \Sigma_p(K_m,n), \mathbb{Z}_2)$ is a cyclic $\mathbb{F}_2 [\mathbb{Z}]$-module, and hence a cyclic $\mathbb{F}_2[\mathbb{Z}_p]$- module as well. 
So
$H_1( \Sigma_p(K_{m,n}), \mathbb{Z}_2) \cong \mathbb{F}_2[t]/ q(t)$ for some $q(t)$ dividing $\sum_{i=0}^{p-1} t^i$. Finally, note that $q(t)\neq1$, since we can compute from the Alexander polynomial that $H_1(\Sigma_p(K_{m,n}))$ has nontrivial 2-torsion.  Therefore, since $\sum_{i=0}^{p-1} t^i$ is irreducible in $\mathbb{F}_2 [\mathbb{Z}]$ we can conclude that $H_1( \Sigma_p(K_{m,n}), \mathbb{Z}_2)$ is isomorphic to $V_p= \mathbb{F}_2[t] \Big/  \sum_{i=0}^{p-1} t^i$, which is an irreducible $\mathbb{F}_2[\mathbb{Z}_p]$-module.  
\end{proof} 

 An identical argument shows that $H_1(\Sigma_p(K_{m,n}^-), \mathbb{Z}_2) \cong V_p$ is irreducible whenever 2 is a primitive root mod $p$ and $p$ does not divide $2n(2n-1)$. 
\medskip

The computational simplifications of \cite{HKL10} discussed in Lemma \ref{covercomp} require that we choose some nonzero equivariant homomorphism $\rho: \pi_1(X_p(K_{m,n})) \to V_p$  and extend $\epsilon \times \rho$ to $\tilde{\rho}: \pi(X(K_{m,n})) \to \mathbb{Z} \ltimes V_p $ with $\tilde{\rho}(\mu)=(x,0)$, where $\mu$ is a preferred meridian in $\pi(X(K_{m,n}))$. Note that any equivariant $\rho$ must factor through $H_1(\Sigma_p(K))$, since it satisfies $\rho(\mu^p)= \rho(\mu \mu^p \mu^{-1})= t \cdot \rho(\mu^p)$. We will instead directly construct $\tilde{\rho}$.

A Wirtinger presentation for the knot group of $K_{m,n}= P(2n, 2k+1, -2n-1,-2k-1)$ is given by the following, where here $a \cdot b$ denotes $ aba^{-1}$. 
\begin{align*} 
\left\{
x_1, \cdots, x_{4n+4k+3}: 
\begin{array}{ll}
x_{i+1} = x_{ i+3n+3k+3} \cdot x_i, &  i=1, \dots, k \\
x_{i+1} = x_{i+2n+2k+2} \cdot x_i, & i=k+1, \dots, n+k+1 \\
x_{i}= x_{i+n+k+1} \cdot x_{i+1}, &  i= n+k+2, \dots, n+2k+2\\
x_{i}= x_{i+n+2k+2} \cdot x_{i+1}, & i= n+2k+3, \dots, 2n+2k+1\\
x_{2n+2k+2} = x_1 \cdot x_{2n+2k+3} &\\
x_{i}= x_{i-(n+k)} \cdot x_{i+1}, & i=2n+2k+3, \dots, 2n+3k+2\\
x_{i+1}= x_{i-(2n+2k+1)} \cdot x_i, & i= 2n+3k+3, \dots, 3n+3k+2 \\
x_{i+1}= x_{i-(3n+3k+2)} \cdot x_i, & i = 3n+3k+3, \dots, 3n+4k+3 \\
x_{i}= x_{i-(2n+2k+1)} \cdot x_{i+1} & i= 3n+4k+4, \dots 4n+4k+2\\
x_{4n+4k+3}= x_{2n+2k+2} \cdot x_1 & \\
\end{array}
\right.
\end{align*}

We choose as preferred meridian $\mu= x_{1}$. Note that since $\tilde{\rho}$ extends some $\epsilon \times \rho$, we must have  $\tilde{\rho}(x_i)= (x, v_i)$ for each of the Wirtinger generators.  The Wirtinger relation $x_l=x_i \cdot x_j$ implies that $v_l= (1-t) v_i + t v_j$ in $V_p$. 
After some simple reductions of the linear relations coming from the above Wirtinger presentation, we see that $\tilde{\rho}$ is determined by our choice of $v_1, v_{k+1}, v_{n+k+2},$ $v_{n+2k+3}, v_{2n+2k+3}$,$ v_{2n+3k+3}, v_{3n+3k+3}, \text{ and } v_{3n+4k+4}$. In addition, any choice satisfying 
$v_1=v_{n+2k+3}= v_{2n+2k+3}= v_{3n+4k+4}$ and
$v_{k+1}=v_{n+k+2}=v_{2n+3k+3}=v_{3n+3k+3}$ determines a valid $\tilde{\rho}$. 

Since we require that $\tilde{\rho}(\mu)= \tilde{\rho}(x_{1})= (x,v_1)=(x,0)$, the map $\tilde{\rho}$ is entirely determined by our choice of $a= v_{k+1}$.\footnote{
Note that when $n$ does not satisfy our divisibility requirements with regards to $p$, the map  described above is still a homomorphism, but there are many other choices.}
In fact, since we will also choose
 $\chi: V_p \to \mathbb{Z}_2$, there are essentially only two distinct choices of $\tilde{\rho}$: the trivial map with $a=0$ and the map corresponding to $a=1$. We will choose $a=1$. 

\noindent
We will also choose\footnote{Note that this is a significant choice: for $p>3$, sample computations indicate that different choices of $\chi$ give very different twisted Alexander polynomials.} $\chi: V_p \to \mathbb{Z}_2$ by 
$
 \chi(t^i)= \left\{
  \begin{array}{ll}
       1 & \text{ if } i=0,2\\
       0 & \text{ else}
     \end{array}
\right.$, and define $\rho_\chi: \pi_1(X_p(K)) \to \mathbb{Z}_2$ as the composition 
\[\rho_\chi: \pi_1(X_p(K)) \xrightarrow{ab.} H_1(X_p(K)) \xrightarrow{i_*} H_1(\Sigma_p(K)) \xrightarrow{\rho} V_p \xrightarrow{\chi} \mathbb{Z}_2.\]

Therefore, by Theorem \ref{covercomp} we have that $\Delta_{X_p(K), \epsilon \otimes \rho_\chi}(t) = \Delta_{X(K), \Phi}(t)$, where $\Phi: \pi_1(K) \to GL_p(\mathbb{Q}[t^{\pm1}]$ is defined by
\[
\begin{array}{rl}
x_{n+k+3}, \dots, x_{2n+3k+2},  x_{3n+4k+4}, \dots, x_{4n+4k+3}, x_1 \mapsto x \\
 x_2, \dots, x_{n+k+2}, x_{2n+3k+3}, \dots, x_{3n+4k+3} \mapsto y \\
 \end{array}, \text{where }
\]
\begin{align*}\label{xydef}
 x= \left[
\begin{array}{cccccc}
0 & 1 & 0& 0 & \cdots &0\\
0 & 0 & 1 & 0 & \cdots &0 \\
0& 0 & 0 & 1 & \cdots &0 \\
\cdots & \cdots& \cdots &\cdots & \ddots & \vdots \\
 0 & 0 &0 &0 & \cdots &1 \\
 t & 0 & 0 & 0 &0 &0 
\end{array}
\right]_{p \times p} \text{  } 
  y= \left[
\begin{array}{cccccc}
0 & 1 & 0& 0 & \cdots &0\\
0 & 0 & -1 & 0 & \cdots &0 \\
0& 0 & 0 & 1 & \cdots &0 \\
\cdots & \cdots& \cdots &\cdots & \ddots & \vdots \\
 0 & 0 &0 &0 & \cdots &1 \\
 -t & 0 & 0 & 0 &0 &0 
\end{array}
\right]_{p \times p}. 
\end{align*}
Note that an almost identical construction gives maps $\tilde{\rho}^*: \pi_1(K_{m,n}^-) \to \mathbb{Z} \ltimes V_p$ and $\chi^*: V_p \to \mathbb{Z}_2$.

\vspace{.5cm}

\subsection{Computation of the reduced twisted Alexander polynomial}

First, recall that the twisted Alexander polynomial is only well defined up to units in $\mathbb{Q}[t^{\pm1}]$, and so we let $\doteq$ denote equality up to multiplication by units and frequently omit factored-out powers of $t$. 

\begin{lem}\label{comp}
Let $m=2k+1,n,p \in \mathbb{N}$ be such that  $p$ divides $m$. Suppose that $n \geq \frac{p+1}{2}$ and that $p$ does not divide $2n(2n+1)$. So $2n=bp+a$ for some $0<a<p-1$ and $b\geq 1$.\footnotemark 

 Then, with $\rho$ and $\chi$ as above, the reduced twisted Alexander polynomial for $K_{m,n}$ is given by 
$ \tilde{\Delta}_{m,n}(t)= f_{b}(t) g_n(t) h_k(t)^2 (t-1)^{-2} $
where $h_k(t) \in \mathbb{Z}[t]$ and 
\begin{align*}
 f_{b}(t)&:= 2\sum_{i=0}^{2b} t^i + t^{b} = 2t^{2b}+ 2t^{2b-1} + \cdots + 2t^{b+1} + 3t^{b}+2t^{b-1} + \cdots + 2,\\
 g_n(t)&:= (4a-6) \sum_{i=0}^{2b}(-t)^i + (-t)^b -4(p-4)t\left(\sum_{i=0}^{b-1} (-t)^i \right)^2.
\end{align*}
\end{lem}
\footnotetext{Note that this computation does not depend on 2 being a primitive root mod $p$, though it does use the divisibility relations between $p, m, \text{and } n$ and that $n \geq \frac{p+1}{2}$.  In particular, this formula does give non-norm reduced twisted Alexander polynomials for many $K_{m,n}$ not satisfying the conditions of Theorem \ref{maintheorem}-- for example, for $K=P(8,7,-9,-7)$. 
 However, when 2 is not primitive mod $p$, this is not enough to obstruct sliceness for $K$.
 } 

As usual, an analogous result holds for $K_{m,n}^-=P(2n,m,-2n+1,-m)$.One key difference, though, is that instead of $f_b(t)$ as above we have $f_b^*(t)= 2\sum_{i=0}^{2b} t^i -t^{b}=2t^{2b}+ 2t^{2b-1} + \cdots + 2t^{b+1} + t^{b}+2t^{b-1} + \cdots + 2$. There is also a slightly different $g^*_n(t)$. 

\begin{proof}

First note that by Lemmas \ref{covercomp} and \ref{foxder} that if we let $Z$ be the reduced Fox derivative matrix of a reduced Wirtinger presentation for $\pi_1(X(K))$ then 
$
\widetilde{\Delta}_{X_p(K), \epsilon \otimes \rho_\chi}(t) = \widetilde{\Delta}_{X(K), \Phi}(t)= \Delta_{X(K), \Phi}(t) (t-1)^{-1} = \det(\Phi(Z)) (t-1)^{-2}.
$
So it suffices to show that $\det(\Phi(Z)) \doteq f_b(t) g_n(t) h_k(t)^2$ as defined above. 

We will use the following simplification of our original Wirtinger presentation:\footnote{Note that  $a= x_{2k+2n+3}, b=x_{k+n+2}, c=x_{3k+3n+3}, e=x_{1}$, $\alpha=x_{2k+n+3}, \beta= x_{3k+2n+3}, \gamma=x_{k+1},$ and $\eta=x_{4k+3n+4}$. So $\Phi(a)= \Phi(e)= \Phi(\alpha) = \Phi(\eta)=x$   
and $\Phi(b)= \Phi(c)= \Phi(\beta)= \Phi(\gamma)= y$.}
\[ \pi_1(K)=  \left\{ 
\begin{array}{l}
a,b,c,e \\
 \alpha, \beta, \gamma, \eta
 \end{array} 
  \text{ s.t.} 
\begin{array}{ll}
a= (\eta \alpha)^{-n} \alpha (\eta \alpha)^n  & e=(\eta\alpha)^{-(n-1)} \alpha^{-1}(\eta \alpha)^{n}\\
b=(\beta \gamma)^n \beta (\beta \gamma)^{-n} & c= (\beta \gamma)^{n+1}  \beta^{-1} (\beta \gamma)^{-n}\\
\gamma= (ec)^k e (ec)^{-k} & \eta= (ec)^{k+1}e^{-1}(ec)^{-k} \\
\beta= (ba)^{-k}a (ba)^k  
\end{array} \right.
\]  

The Fox derivatives of these relations are given by 
\[  \begin{array}{l}
(\eta \alpha)^n d a 
 + \left[ (1-a) \sum_{i=0}^{n-1}(\eta\alpha)^i \eta -1\right] d \alpha 
 + \left[(1-\alpha)\sum_{i=0}^{n-1} (\eta \alpha)^i  \right] d \eta, \\

\alpha(\eta \alpha)^{n-1} de 
 +\left[ (1 - \eta) \sum_{i=0}^{n-1} (\alpha \eta)^i  \right] d \alpha  
 + \left[(\alpha-1) \sum_{i=0}^{n-2} (\eta \alpha)^i  -  (\eta \alpha)^{n-1} \right] d \eta,\\

db
+ \left[(b-1) \sum_{i=0}^n (\beta \gamma)^i \right] d \beta 
+ \left[ (b-1) \sum_{i=0}^{n-1} (\beta \gamma)^i \beta - (\beta \gamma)^n \beta \right] d \gamma,\\

dc 
+ \left[ (c-1) \sum_{i=0}^{n-1} (\beta \gamma)^i - (\beta \gamma)^n \right] d \beta 
+ \left[ (c-1) \sum_{i=0}^{n-1} (\beta \gamma)^i \beta\right] d\gamma, \\

d \gamma
+ \left[ (\gamma-1)\sum_{i=0}^{k-1} (ec)^i - (ec)^k \right] de
 + \left[ (\gamma-1)\sum_{i=0}^{k-1} (ec)^ie \right] dc, \\

d \eta 
+ \left[ (\eta-1) \sum_{i=0}^k (ec)^i \right] de
 + \left[ (\eta-1) \sum_{i=0}^{k-1} (ec)^i e -(ec)^ke \right] dc, \text{ and}
\\

(ba)^k d \beta 
+ \left[ (1-a) \sum_{i=0}^{k-1} (ba)^i \right] db
 + \left[-1 + (1-a) \sum_{i=0}^{k-1} (ba)^ib \right] da. \\
\end{array}
\]

So the image of the reduced Fox derivative matrix (with column corresponding to $e=\mu$ deleted) is  $\Phi(Z)= \left[\Phi(Z)_L \,\,\, \Phi(Z)_R \right]$, where

\begin{align*}
\Phi(Z)_L=
\left[
\begin{array}{ccc}
x^{2n} &  0 & 0\\
0 & 0 &0 \\
0&1&0\\
0&0&1 \\
0&0& (y-1) \sum_{i=0}^{k-1} (xy)^ix \\
0& 0 & (x-1) \sum_{i=0}^{k-1} (xy)^ix - (xy)^kx \\
y\sum_{i=0}^{k-1}(xy)^i-\sum_{i=0}^k(xy)^i  & (1-x) \sum_{i=0}^{k-1}(yx)^i &0\\
\end{array}
\right]
\end{align*}

\begin{align*}
\Phi(Z)_R=
\left[
\begin{array}{cccc}
 - \sum_{i=0}^{2n} (-x)^i & \sum_{i=0}^{2n-1} (-x)^i &0 & 0 \\
 \sum_{i=0}^{2n-1} (-x)^i & -\sum_{i=0}^{2n-2}(-x)^i & 0&0\\
0 &0 &-\sum_{i=0}^{2n+1}(-y)^i & \sum_{i=1}^{2n+1}(-y)^i\\
0 &0 & - \sum_{i=0}^{2n}(-y)^i & \sum_{i=1}^{2n}(-y)^i \\
 0&0&0&1 \\
 0&1&0&0 \\
0 &0& (yx)^k &0\\
\end{array}
\right]
\end{align*}

\noindent The matrix $\Phi(Z)=\left[ \Phi(Z)_L \,\,\, \Phi(Z)_R \right]$
 can be shown via simple row and column moves  to have the same determinant (up to units) as the matrix 
\[
\widehat{\Phi(Z)}=
\left[
\begin{array}{ccc}
-A_{n} & 0 & B_{n} \\
y(xy)^k A_{n} & C_k & 0 \\
D_{k,n} & C_k & E_k
 \end{array} \right],
\]
\begin{align*}
 &\text{ where }
A_n= -\sum_{i=0}^{2n}(-y)^i, \,
 B_n= \sum_{i=0}^{2n-1} (-x)^i, \,C_k= 1+ (y-1)x\sum_{i=0}^{k-1} (yx)^i,\,  \\
&D_{k,n}= \sum_{i=0}^{2n+1} (-y)^i + (y-1)x \sum_{i=0}^{k-1} (yx)^i (-y)^{2n+1} , \, E_k= 1+(x-1)y \sum_{i=0}^{k-1} (xy)^i.
\end{align*}

\noindent Observe that 
\begin{align*}
 \det( \widehat{\Phi(Z)} &= \det(C_k) \det 
\left[
\begin{array}{ccc}
-A_{n} & 0 & B_{n} \\
y(xy)^k A_{n} & I  & 0 \\
D_{k,n} & I & E_k
\end{array}
\right] 
\\
&= \det(C_k) \det \left[
\begin{array}{cc}
-A_n & B_n \\
D_{k,n} - y(xy)^k A_{n} & E_k 
\end{array}
\right]
\\
&= \det(C_k) \det(E_k)
\det \left[
\begin{array}{cc}
-A_n & B_n \\
E_k^{-1} \left(D_{k,n} - y(xy)^k A_{n}\right) & I 
\end{array}
\right]
\\
&= \det(C_k) \det(E_k) \det (-A_n - B_n E_k^{-1}(D_{k,n}- y(xy)^kA_n)).
\end{align*}
By Lemma \ref{ckek}, $\det(C_k)= \det(E_k)$. Let $h_k(t):= \det(C_k)= \det(E_k)$, so 
\begin{align*}
\det( \widehat{\Phi(Z)}) \doteq h_k(t)^2 \det (A_n +B_n E_k^{-1}(D_{k,n}- y(xy)^kA_n)) .
\end{align*}
Note that the entries of $C_k$ are in $\mathbb{Z}[t]$, so $h_k(t) \in \mathbb{Z}[t]$.
By Lemma \ref{ckek}, we also have that the matrix $E_k^{-1}(D_{k,n}- y(xy)^kA_n))$ is independent of $k$. 
So let $k_0:= \frac{p-1}{2}$ and $F_n:= E_{k_0}^{-1}(D_{k_0,n}- y(xy)^{k_0}A_n))$. Then

\begin{align}\label{phihat}
 \det( \widehat{\Phi(Z)}) \doteq h_k(t)^2 \det (A_n +B_n F_n). 
\end{align}

Now, recall that  $2n, 2n+1 \not \equiv 0 \mod p$ and so we can write $2n=b p + a$ for $0<a<p-1$.
By Lemma \ref{matrixform}, we have that 
$\det(A_n+B_n F_n)\doteq f_b(t) \det g_n(t)$, where $f_b(t)$ is as above, $\Psi_b(t)= (-1)^b t \left(2 \sum_{i=0}^{2b} (-t)^i + (-t)^b \right)$,
 and
\begin{align*}
g_n(t)&:= \det( G_n) (1+t)^{-1}= \det \left[
\begin{array}{ccc}
 (p-a-2)\beta_{b}(t)&-1&-1\\
\Psi_b(t) & 2(-1)^b & - 2t^{b+1}
\\
(a-2) \beta_{b+1}(t)&1&-t
\end{array}
\right](1+t)^{-1},\\
\end{align*}

\noindent Observe that
\begin{align*}
 g_n(t)(t+1)&=- \Psi_b(t)(t+1)+ 2t(p-a-2) \beta_b(t) ((-1)^{b+1} + t^{b})
+ 2(a-2) \beta_{b+1}(t) (t^{b+1}+(-1)^{b}) \\
&\doteq(t+1) \left[ 2\sum_{i=0}^{2b} (-t)^i+ (-t)^b -4(p-a-2)t \left(\sum_{i=0}^{b-1} (-t)^i \right)^2 + 4(a-2) \left(\sum_{i=0}^{b} (-t)^i \right)^2  \right] \\
& \doteq (t+1) \left[  2\sum_{i=0}^{2b} (-t)^i+(-t)^b - 4(p-4)t\left(\sum_{i=0}^{b-1} (-t)^i \right)^2 + 4(a-2) \sum_{i=0}^{2b} (-t)^i \right]
\end{align*}
and so 
\begin{align}\label{gnform}
g_n(t)= (4a-6) \sum_{i=0}^{2b}(-t)^i + (-t)^b -4(p-4)t\left(\sum_{i=0}^{b-1} (-t)^i \right)^2.
\end{align}

\noindent Therefore, combining (\ref{phihat}),  (\ref{gnform}),  and Lemma \ref{matrixform}  we have as desired that 
\begin{align*}
\widetilde{\Delta}_{m,n}(t)= \det(\widehat{\Phi(Z)}) (t-1)^{-2}= h_k(t)^2 f_b(t) g_n(t) (t-1)^{-2}.
\end{align*} 
\end{proof}


\subsection{$\widetilde{\Delta}_{m,n}(t)$ is not a norm.}
\vspace{.1cm}

 We will now show that $\widetilde{\Delta}_{m,n}(t)$ is not a norm in $\mathbb{C}[t^{\pm1}]$ and hence is certainly not a norm in any $\mathbb{Q}(\xi_{2^n})[t^{\pm1}]$. 

\begin{thm}\label{notnorm}
Let $m,n, p \in \mathbb{N}$ be such that $p$ divides $m$ but $p$ does not divide $2n(2n+1)$, and such that $(n,p) \neq (3,5)$. 
 Let $f_{b}(t), g_n(t)$ be as above and $h_k(t) \in \mathbb{Z}[t]$. 
Then $\tilde{\Delta}_{m,n}(t)= f_{b}(t) g_n(t) h_k(t)^2 (t-1)^{-2}$ is not a norm in $ \mathbb{C}[t^{\pm1}]$. 
\end{thm}

\begin{proof}
First, observe that our map $\rho: \pi_1(X_p(K_{m,n})) \to \mathbb{Z}_2 \hookrightarrow \mathbb{Q}^x \cong GL_1(\mathbb{Q})$ is trivially unitary. By Corollary 5.2 of  \cite{KL99a}, the corresponding reduced twisted Alexander polynomial $\tilde{\Delta}_{m,n}(t)$ is a symmetric polynomial. Therefore, since $f_{b}(t)$ and $g_n(t)$ have symmetric coefficients, $h_k(t)^2$ and hence $h_k(t) \in \mathbb{Z}[t]$ must as well. So $h_k(t)^2= t^{\deg(h_k)} h_k(t) h_k(t^{-1})=  t^{\deg(h_k)} h_k(t) \overline{ h_k(t^{-1})}$
is a norm, as is $(t-1)^{-2}$. 
 So it suffices to show that $f_{b}(t) g_n(t)$ is not a norm.

Note that both $g_n(t)$ and $f_b(t)$ are of degree $2b$, and so we can check by explicitly computing the three highest-degree coefficients of each polynomial that for $(n,p) \neq (3,5)$, the polynomial $g_n(t)$ is not a multiple of $f_b(t)$. 
Therefore, our result will follow from showing that $f_{b}(t)$ is irreducible in $\mathbb{Q}[t]$ and not a norm in $\mathbb{C}[t]$, as is checked in Lemma \ref{irredcheck}. 
\end{proof}

We need the following result of P. Lakatos, which describes when perturbations of certain products of cyclotomic polynomials have only unit norm roots. 

\begin{thm} [\cite{Lak02}] \label{unitroot} 

Suppose that $p(z) \in \mathbb{R}[z]$ is such that there are $l, a_0, \cdots, a_{ \left \lfloor{\frac{r}{2}}\right \rfloor} \in \mathbb{R}$ and $r \geq 2$ with
\[ p(z)= l(z^r+ z^{r-1} + \dots + z + 1) + \sum_{k=1}^{ \left \lfloor{\frac{r}{2}}\right \rfloor} a_k (z^{r-k}+z^k).\]
If $|l| \geq 2 \sum_{k=1}^{ \left \lfloor{\frac{r}{2}}\right \rfloor} |a_k|$, then $p(z)$ has all roots on the unit circle. 
\end{thm}

\begin{lem}\label{irredcheck}
For any $b \in \mathbb{N}$, the polynomial $f_b(t)= 2\sum_{i=1}^{2b} t^i + t^b$ 
is irreducible over $\mathbb{Q}[t]$ and not a norm in $\mathbb{C}[t]$. 
\end{lem}

\begin{proof}
First, observe that $f_{b}(t)$ satisfies the hypotheses of Theorem \ref{unitroot}, since we have $l=2$,  $a_k=0$ for $k=0, \dots, b-1$, and $a_b=1$. So for any $b \in \mathbb{N}$, the polynomial $f_b(t)$ has all of its roots on the unit circle.

Since $f_b(t)$ is symmetric, there is $l_b(t) \in \mathbb{R}[t]$ such that $f_b(t)= l_b\left(t+ \frac{1}{t}\right)$. 
However, since $f_b(t)$ has only unit norm roots, any factor of $f_b(t)$ over $\mathbb{Q}[t] \subset \mathbb{R}[t]$ must be symmetric, and so of the form $g\left(t+ \frac{1}{t} \right)$ for some $g(t)$ dividing $l_b(t)$. 
In particular, in order to show that $f_b(t)$ is irreducible in $\mathbb{Q}[t]$ it suffices to show that $l_b(t)$ is irreducible in $\mathbb{Q}[t]$.
Now note that $l_b(t) = \sum_{j=0}^{k=b} a_j t^j$  must have $a_b=2$, $a_j$ even for $0<j<b$, and $a_0$ odd. Therefore, by Eisenstein's criterion with $p=2$ and Gauss's Lemma,  the integral polynomial $t^b l_b(t^{-1})$ is irreducible over $\mathbb{Q}[t]$, and hence $l_b(t)$ and $f_b(t)$ are as well.  Since $f_b(t)$ is irreducible, its roots are distinct-- in particular, $f_b(t)$ has at least one complex root of unit norm with odd multiplicity. 

Now let $ t^k g(t)  \overline{g(t^{-1})}$ be any norm in $\mathbb{C}[t]$.  Note that if $\alpha$ is a nonzero root of $g(t)$ then $\overline{\frac{1}{\alpha}}$ is a root of $\overline{g(t^{-1})}$. In particular, if $\alpha$ is a unit norm root of $g(t)$, then $\alpha= \overline{\frac{1}{\alpha}}$ is a root of $\overline{g(t^{-1})}$ of the same multiplicity. That is, any norm in $\mathbb{C}[t]$ must have all unit-norm roots occurring with even multiplicity, and so $f_b(t)$ is not a norm. 
\end{proof}

Almost identical arguments show that $f^*_b(t)$ is irreducible, relatively prime to $g_n^*(t)$, and not a norm, and hence that the reduced twisted Alexander polynomial for $K_{m,n}^-$  constructed via $\tilde{\rho}^*$ and $\chi^*$ is not a norm in $\mathbb{C}[t^{\pm1}]$.


\section{Matrix computations}

The remaining results are primarily consequences of matrix manipulation. 

\begin{lem} \label{ckek}
Let $k= \frac{p+1}{2}+jp$ and $n \in \mathbb{N}$, and let  $A_n, C_k, D_{k,n},$ and $E_k$ be as before. 
Then the following hold:
\begin{enumerate}
\item $\det(E_k)= \det(C_k)$
\item $F_{k,n}:= E_k^{-1}(D_{k,n}-y(xy)^kA_n)$ is independent of $k$. 
\end{enumerate}
\end{lem}

\begin{proof} 
First, observe that $y=axa$, where $a$ is a diagonal matrix with entries 
$a_{i,i} = \left\{
 \begin{array}{cc} -1 & \text{if } i=1,2 \\
1 & \text{else}
\end{array} \right.$. Therefore, 
$ (xy)^{ \frac{p-1}{2}} x= (xaxa)^\frac{p-1}{2} x= (xa)^p a$ and
$ y(xy)^\frac{p-1}{2} = axa(xaxa)^\frac{p-1}{2}= a (xa)^p$. 
Since $(xa)^p$ can be easily computed to be the diagonal matrix $t I_{p}$, we have that $a(xa)^p=(xa)^pa$ and hence $ (xy)^{ \frac{p-1}{2}} x= y(xy)^\frac{p-1}{2}$.
It also follows that  $(xy)^{ip}= (xa)^{2ip}= t^{2i} I_{p}= (yx)^{ip}$ for any $i \in \mathbb{N}$.
Therefore, recalling that $k= jp + \frac{p-1}{2}$, we have
\[
(xy)^k x = (xy)^{jp}  (xy)^{\frac{p-1}{2}} x = t^{2j}  (xy)^{\frac{p-1}{2}} x =   y(xy)^{\frac{p-1}{2}} t^{2j}= y(xy)^{\frac{p-1}{2}} (yx)^{jp}= y(xy)^k.
\]

\noindent Now  observe that 
\begin{align*}
E_k (1-xy) &= 1- xy + (x-1)y(1-(xy)^k)= 1+ y(xy)^k-y - (xy)^{k+1} \\
&=  1+ y(xy)^k -(1+(xy)^kx)y = 1+y(xy)^k - (1+y(xy)^k)y \\
&= (1+ y(xy)^k) (1-y). 
\end{align*}Similarly,
\begin{align*}
C_k (1-yx) &= 1+x(yx)^k-x -(yx)^{k+1} =  1+ x(yx)^k - (1+y(xy)^k) x \\
&= 1+ y(xy)^k- (1+y(xy)^k)x = (1+y(xy)^k) (1-x).
\end{align*}

The matrices $x$ and $y$ are invertible, and so $\det(1-xy)= \det(1-yx)$. 
We can also explicitly check that $\det(1-xy) \neq 0$ and $\det(1-x)= \det(1-y)$, and therefore conclude that $\det(C_k)= \det(E_k)$. 

It also follows that $E_k^{-1} C_k$  and $E_k^{-1} (1+ y(xy)^k)$ are independent of $k$, since by the above
\begin{align*}
E_k^{-1} C_k &= (1-xy) (E_k (1-xy))^{-1} C_k (1-yx)(1-yx)^{-1} \\
&= (1-xy) (1-y)^{-1} (1+y(xy)^k)^{-1} (1+y(xy)^k) (1-x)(1-yx)^{-1} \\
&= (1-xy)(1-y)^{-1} (1-x)(1-yx)^{-1}.
\end{align*}
and 
\begin{align*}
E_k^{-1}(1+y(xy)^k)&= (1-xy) (1-y)^{-1} (1+y(xy)^k)^{-1}(1+y(xy)^k)
\\& = (1-xy)(1-y)^{-1}.
\end{align*}

\noindent Finally, observe that 
\begin{align*} D_{k,n}- y(xy)^kA_n &= \sum_{i=0}^{2n+1}(-y)^i + (y-1)x\sum_{i=0}^{k-1} (yx)^i (-y)^{2n+1} + y(xy)^k  \sum_{i=0}^{2n}(-y)^i \\
&= (1+y(xy)^k) \sum_{i=0}^{2n} (-y)^i + \left( 1+  (y-1)x\sum_{i=0}^{k-1} (yx)^i \right) (-y)^{2n+1} \\
&=- (1+y(xy)^k) A_n - C_k y^{2n+1}.
\end{align*}
Therefore,
$F_{k,n}= E_k^{-1} ( D_{k,n}- y(xy)^kA_n) =- E_k^{-1}(1+y(xy)^k) A_n -E_k^{-1} C_k y^{2n+1} 
$ 
is independent of $k$ as well. 
\end{proof}


\begin{lem}\label{matrixform}
Let $p$ be prime and $n \in \mathbb{N}$ such that $2n=bp+a$ for $0<a<p-1$ and $b \geq 1$. 
Then $\det(A_n+ B_n F_n) \doteq f_b(t) \det(G_n)(1+t)^{-1}$, where $f_b(t)$ is as in Lemma \ref{comp} and
\begin{align*}
G_n:= \left[
\begin{array}{ccc}
 (p-a-2)\beta_{b}(t)&-1&-1\\
\Psi_b(t) & 2(-1)^b & - 2t^{b+1}
\\
(a-2) \beta_{b+1}(t)&1&-t
\end{array}
\right]
\end{align*}
\begin{align*}
\beta_{b}(t)=2 \sum_{i=1}^b (-t)^i , \, \Psi_b(t)= (-1)^b t \left(2 \sum_{i=0}^{2b} (-t)^i + (-t)^b \right)
 \end{align*}
\end{lem}

\begin{proof}
First, observe that when $p=3$ or $p=5$ $A_n+B_nF_n$ is of small size, and one can explicitly compute the form above, with minimal simplification required. So suppose $p \geq 7$. 
Observe that $A_n(1+y)= -(1+ y^{2n+1})$, so we will begin by considering the matrix
\begin{align*}
-\left( A_n + B_n F_n \right) (1+y) =- A_n + B_n E_0^{-1} ( C_0 y^{2n+1} + (1+y(xy))^{\frac{p-1}{2}} A_n) (1+y) \\
= 1+ y^{2n+1} +B_n \left( E_0^{-1}C_0 (1+y)y^{2n+1} - E_0^{-1} (1+y(xy)^{k_0})(1+y^{2n+1}) \right)
\end{align*}

We can compute $E_0^{-1} C_0$ and $E_0^{-1} (1+y(xy)^{k_0})$ using the expressions from Lemma \ref{ckek}. Also note that $ (1+x)B_n= 1- x^{2n}$ is also easily computable, leading us to an easy verification for the form of $B_n$.

 Combining these expressions, we get the following form for $(-1)\left(A_n + B_n F_n \right) (1+y)$ when $1<a<p-2$. Note that similar expressions hold for $a=1$ and $a=p-2$.

\begin{landscape}
The matrix $-\left(A_n + B_n F_n \right) (1+y)$, written as a block matrix with dimensions 
$(2+(p-a-2)+2+(a-2), 2+(a-1)+2+(p-a-3))$.

\begin{align*}
\footnotesize
 \left[
\begin{array}{c|c|c|c}
\begin{array}{cc}
-\gamma_b(t) & - \gamma_b(t) \\
\gamma_b(t) & \gamma_b(t)\\
\end{array}
&
\begin{array}{cccc}
\alpha_b(t)    &\cdots & \alpha_b(t) & \alpha_b(t) - t^b \\
-\alpha_b(t)&  \cdots & -\alpha_b(t)& -\alpha_b(t) 
\end{array}
&
\begin{array}{cc}
\epsilon_b(t) & \epsilon_b(t) - t^b \\
- \epsilon_b(t) & -\epsilon_b(t)+2t^b
\end{array}
&
 \begin{array}{cccc}
\alpha_{b-1}(t) & \alpha_{b-1}(t)& \cdots & \alpha_{b-1}(t) \\
-\alpha_{b-1}(t) & -\alpha_{b-1}(t) & \cdots & - \alpha_{b-1}(t)
\end{array}
\\ \hline
\begin{array}{cc}
\beta_b(t) & \beta_b(t) \\
-\beta_b(t) & - \beta_b(t) \\
\vdots & \vdots \\
\beta_b(t) & \beta_b(t)\\
- \beta_b(t) + t^{b+1} & - \beta_b(t) 
\end{array}
&
\text{ \Huge{0}}_{p-a-2,a-1} 
&
\begin{array}{cc}
t^b \beta_b(t) & t^b \beta_b(t) - t^b \\
-t^b \beta_b(t) & - t^b \beta_b(t) \\
\cdots & \cdots \\
t^b \beta_b(t) & t^b \beta_b(t) \\
-t^b \beta_b(t) & - t^b \beta_b(t)
\end{array}
&
\begin{array}{cccc}
\quad -t^b \qquad & 0 \quad& \cdots & \quad 0 \quad \\
\quad -t^b \qquad &  -t^b \quad & \cdots & \quad 0 \quad\\
\quad 0 \qquad& \ddots & \ddots & 0\\
\quad 0 \qquad & \cdots & -t^b& -t^b \\
\quad 0 \qquad& \cdots & 0 & -t^b \\
\end{array}
\\ \hline
\begin{array}{cc}
\eta_b(t) & \eta_b(t) + 2t^{b+1} \\
-\eta_b(t) -t^{b+1}& - \eta_b(t) - 2t^{b+1} 
\end{array}
&
\begin{array}{cccc}
-\phi_b(t) & -\phi_b(t) & \cdots  & -\phi_b(t) 
\\
\phi_b(t) - t^{b+1} & \phi_b(t)&  \cdots  & \phi_b(t) 
\end{array} 
&
\begin{array}{cc}
-4t^{2b} & -4t^{2b} \\
4t^{2b} & 4t^{2b}
\end{array}
&
\begin{array}{cccc}
-\theta_b(t) \,\,  &\quad -\theta_b(t) &\quad \cdots  & \quad -\theta_b(t)
\\
\theta_b(t) \,\, &\quad\theta_b(t) & \quad \cdots  & \quad \theta_b(t)
\end{array}
\\ \hline
\begin{array}{cc}
\beta_{b+1}(t) & \beta_{b+1}(t) \\
-\beta_{b+1}(t) & -\beta_{b+1}(t)\\
\vdots & \vdots \\
\beta_{b+1}(t) & \beta_{b+1}(t) \\
\end{array}
&
\begin{array}{ccccc}
-t^{b+1} & -t^{b+1}  & 0 & \cdots & 0\\
0 & t^{b+1} & -t^{b+1} & \cdots & 0\\
\vdots & \ddots & \ddots & \ddots & \vdots \\
0 & 0 & \cdots & - t^{b+1} & - t^{b+1} \\
\end{array}
&
\begin{array}{cc}
t^b\beta_{b+1}(t) & t^b\beta_{b+1}(t) \\
-t^b\beta_{b+1}(t) & -t^b\beta_{b+1}(t)\\
\vdots & \vdots \\
t^b\beta_{b+1}(t) &t^b \beta_{b+1}(t) \\
\end{array}
&
\text{\Huge{0}}_{a-2,p-a-3}
\end{array}
\right] 
\end{align*}

\begin{align*} \text{ where } 
\alpha_b(t)&: = 2 \sum_{i=0}^b t^i, \, \beta_b(t) := 2 \sum_{i=1}^b (-t)^i,
  \, \gamma_b(t):= 4 \sum_{i=0}^{ \lceil \frac{b}{2} \rceil} t^{2i+1}, \,
  \eta_b(t): = 2 \sum_{i=b+2}^{2b+1}t^i + t^{b+1} + (-1)^{b+1}2 \sum_{i=1}^b (-t)^i, \\
  \epsilon_b(t) &:= (-1)^b 2 \sum_{i=b+1}^{2b} (-t)^i + 3t^b + 2 \sum_{i=0}^{b-1} t^i, \,   \theta_b(t):=  t^{b+1} \alpha_{b-1}(t), \, \text{ and }\phi_b(t):=\theta_b(t) + 2t^{2b+1} = t^{b+1} \alpha_b(t) 
\end{align*}

(Examination of the alternating signs above indicates that the form above applies only when $a$ is odd. The $a$ is even case is exactly analogous and omitted.)

\end{landscape}

\begin{landscape}
Some easy row and column moves\footnote{To be specific, perform the following operations, in this order:
add $r_1$ to $r_2$, 
add $r_{p-a+1}$ to $r_{p-a+2}$,
 add $-c_1$ to $c_2$, 
 add $-c_{a+2}$ to $c_{a+3}$, 
 add $-t^bc_1$ to $c_{a+2}$, 
 add $c_2$ to $c_1$, 
 add $-t^bc_2$ to $c_{a+2}$,
  and add $t^{b+1} r_1$ to $r_{p-a-1}$.}
 let us rewrite this matrix as follows:

\begin{align*} 
\footnotesize
 \left[
\begin{array}{c|c|c|c}
\begin{array}{cc}
-\gamma_b(t) & 0\\
0&0\\
\end{array}
&
\begin{array}{cccc}
\alpha_b(t)    &\cdots & \alpha_b(t) & \alpha_b(t) - t^b \\
0&  \cdots &0& -t^b
\end{array}
&
\begin{array}{cc}
f_b(t)& - t^b \\
0 & t^b
\end{array}
&
 \begin{array}{cccc}
\alpha_{b-1}(t) & \alpha_{b-1}(t)& \cdots & \alpha_{b-1}(t) \\
0 &0 & \cdots & 0
\end{array}
\\ \hline
\begin{array}{cc}
\beta_b(t) & 0 \\
-\beta_b(t) & 0 \\
\vdots & \vdots \\
\beta_b(t) &0\\
- \beta_b(t) &-t^{b+1}
\end{array}
&
\text{ \Huge{0}}_{p-a-2,a-1} 
&
\begin{array}{cc}
0 & - t^b \\
0 & 0\\
\vdots & \vdots \\
0 &0 \\
0& 0
\end{array}
&
\begin{array}{cccc}
\quad -t^b \qquad & 0 \quad& \cdots & \quad 0 \quad \\
\quad -t^b \qquad &  -t^b \quad & \cdots & \quad 0 \quad\\
\quad 0 \qquad& \ddots & \ddots & 0\\
\quad 0 \qquad & \cdots & -t^b& -t^b \\
\quad 0 \qquad& \cdots & 0 & -t^b \\
\end{array}
\\ \hline
\begin{array}{cc}
\Psi_b(t) &  2t^{b+1} \\
0& t^{b+1}
\end{array}
&
\begin{array}{ccccc}
0&0 & \cdots&0  & -t^{2b+1}
\\
- t^{b+1} &0&  \cdots&0  &0 
\end{array} 
&
\begin{array}{cc}
0& -t^{2b+1} \\
0& 0
\end{array}
&
\begin{array}{cccc}
0\,\,  &\quad0 &\quad \cdots  & \quad0
\\
0 \,\, &\quad 0 & \quad \cdots  & \quad0
\end{array}
\\ \hline
\begin{array}{cc}
\beta_{b+1}(t) &0 \\
-\beta_{b+1}(t) & 0\\
\vdots & \vdots \\
\beta_{b+1}(t) &0 \\
\end{array}
&
\begin{array}{ccccc}
-t^{b+1} & -t^{b+1}  & 0 & \cdots & 0\\
0 & t^{b+1} & -t^{b+1} & \cdots & 0\\
\vdots & \ddots & \ddots & \ddots & \vdots \\
0 & 0 & \cdots & - t^{b+1} & - t^{b+1} \\
\end{array}
&
\begin{array}{cc}
0&0 \\
0 &0\\
\vdots & \vdots \\
0 &0\\
\end{array}
&
\text{\Huge{0}}_{a-2,p-a-3}
\end{array}
\right] 
\end{align*}

\begin{align*} \text{ where } 
\alpha_b(t)&: = 2 \sum_{i=0}^b t^i, \, \beta_b(t) := 2 \sum_{i=1}^b (-t)^i,
  \, \gamma_b(t):= 4 \sum_{i=0}^{ \lceil \frac{b}{2} \rceil} t^{2i+1}, \,
  \eta_b(t): = 2 \sum_{i=b+2}^{2b+1}t^i + t^{b+1} +(-1)^{b+1} 2 \sum_{i=1}^b (-t)^i, \\
  \epsilon_b(t) &:= (-1)^b 2 \sum_{i=b+1}^{2b} (-t)^i + 3t^b + 2 \sum_{i=0}^{b-1} t^i, \, \Psi_b(t):= \eta_b(t) + 2t^{b+1} -t^{b+1} \gamma_b(t)
\end{align*}

Note that $f_b(t) = 2\sum_{i=0}^{2b} t^i + t^b$ is obtained in the above matrix as $f_b(t)= \epsilon_b(t) + t^b \gamma_b(t)$. 

\end{landscape}

Therefore, $\det (A_n + B_n F_n) \doteq f_b(t) \det (M_n) (1+t)^{-1}$, where $M_n$ is obtained from the previous matrix by the deletion of rows 1, 2 and columns $p-a+1, p-a+2$ and moving a column. 

\begin{align}
\footnotesize
M_n=
\left[
\begin{array}{c|c|c}
\begin{array}{ccc}
\beta_b(t) & 0 &-t^b \\
-\beta_b(t) & 0 &0 \\
\vdots & \vdots & \vdots\\
\beta_b(t) &0&0\\
- \beta_b(t) &-t^{b+1}&0
\end{array}
&
\text{ \Large{0}}_{p-a-2,a-2} 
&
\text{\Large{E}}_{p-a-2}^b
\\ \hline
\begin{array}{ccc}
\Psi_b(t) &  2t^{b+1} &-2t^{2b+1}
\end{array}
&
\begin{array}{ccc}
0 & \cdots&0  
\end{array}
&\begin{array}{ccc}
0 & \cdots & 0
\end{array}
\\ \hline
\begin{array}{ccc}
0& t^{b+1}&0 \\
\beta_{b+1}(t) &0& 0 \\
\vdots & \vdots &\vdots \\
-\beta_{b+1}(t)& 0 & 0 \\
\beta_{b+1}(t) &0 &-t^{b+1} \\
\end{array}
&
\text{\Large{E}}_{a-1}^{b+1}
&
\text{\Large{0}}_{a,p-a-3}
\end{array}
\right]_{p-2,p-2}
\end{align}

where
\[ E_{k}^b:=
\left[
\begin{array}{ccccc}
-t^b & 0& 0 & \cdots & 0\\
-t^b & -t^b & 0 & \cdots & 0 \\
0 & -t^b & -t^b & \ddots & 0\\
\vdots & \ddots& \ddots & \ddots & \vdots\\
0 & \cdots & 0 & -t^b & -t^b \\
0 & \cdots & 0 & 0 & -t^b\\
\end{array}
\right]_{k,k-1}.
\]

Note that each of the columns $c_4, \dots, c_{p-2}$ of $M_n$ contain exactly two nonzero entries.  We can apply simple row moves to show that $\det(M_n) \doteq \det(G_n)$, where

\begin{align*}
G_n:= \left[
\begin{array}{ccc}
 (p-a-2)\beta_{b}(t)&-1&-1\\
\Psi_b(t) & 2(-1)^b & - 2t^{b+1}
\\
(a-2) \beta_{b+1}(t)&1&-t
\end{array}
\right]
\end{align*}

Finally, note that 

\begin{align*}
\Psi_b(t)&=\eta_b(t) + 2t^{b+1} -t^{b+1} \gamma_b(t) \\
&= 2 \sum_{i=b+2}^{2b+1} t^i + 3t^{b+1} + (-1)^{b+1} 2 \sum_{i=1}^{b} (-t)^i -4 t^{b+1} \sum_{i=1}^{\lceil \frac{b}{2} \rceil} t^{2i+1} \\
&= t \left( 
(-1)^b \sum_{i=b+1}^{2b} 2(-t)^i + 3t^{b} + (-1)^b  \sum_{i=0}^{b-1} 2(-t)^i \right) \\
&= (-1)^b t \left(2 \sum_{i=0}^{2b} (-t)^i + (-t)^b \right), \text{ as desired.}
\end{align*}
\end{proof}

\subsection{Sample computations of $f_b(t)$ and $g_n(t)$.}

Finally, we give some computations of $f_b(t)$ and $g_n(t)$, normalized to have positive leading coefficient. Observe that when $(n,p)=(3,5)$ we have that $f_b(t)= g_n(t)$, and so the associated twisted Alexander polynomial $f_b(t) g_n(t) h_k(t)^2 (t-1)^{-2}$ is a norm. 

\begin{table}[htdp]
\begin{center}
$
\begin{array}{|c|c|c|c|}
n & 2n=bp+a & f_b(t) & g_n(t) \\ \hline
6 & 12 = 1(11)+1 & 2t^2+3t+2 & 2t^2+27t+2 \\
7 & 14 = 1(11)+3 & 2t^2+3t+2 & 6t^2-35t+6 \\
8 & 16 = 1(11)+5 & 2t^2+3t+2 & 14t^2-43t+14 \\
9 & 18 = 1(11)+7 & 2t^2+3t+2 & 22t^2-51t+22\\
10 & 20 = 1(11)+9 & 2t^2+3t+2 & 30t^2-59t+30 \\
12& 24= 2(11)+2 & 2t^4+ 2t^3+ 3t^2+2t+2 & 2t^4-30t^3+59t^2-30t+2 \\
13& 26= 2(11)+4 & 2t^4+ 2t^3+ 3t^2+2t+2 & 10t^4-38t^3+67t^2-38t+10 \\
14& 28= 2(11)+6 & 2t^4+ 2t^3+ 3t^2+2t+2 & 18t^4-46t^3+75t^2-46t+18 \\
\end{array}
$
\end{center}
\caption{Some computations of $f_b(t)$ and $g_n(t)$, with $p=11$}
\label{examples}
\end{table}%

\begin{table}[htdp]
\begin{center}
$
\begin{array}{|c|c|c|c|}
n & 2n=bp+a & f_b(t) & g_n(t) \\ \hline
3& 6= 1(5)+1 & 2t^2+3t+2 & 2t^2+3t+2 \\
4&8=1(5)+3 & 2t^2+3t+2 & 6t^2-11t+6 \\
6& 12= 2(5)+2 & 2t^4+ 2t^3+ 3t^2+2t+2 & 2t^4-6t^3+11t^2-6t+2\\
8& 16= 3(5)+1 & 2t^6+2t^5+2t^4+3t^3+2t^2+2t+2 &2t^6+2t^5-6t^4+11t^3-6t^2+2t+2\\
9& 18= 3(5)+3 & 2t^6+2t^5+2t^4+3t^3+2t^2+2t+2 &6t^6-10t^5+14t^4-19t^3+14t^2-10t+6
\end{array}
$
\end{center}
\caption{More computations of $f_b(t)$ and $g_n(t)$, with $p=5$.}
\label{examples}
\end{table}%

\section*{Acknowledgements}
I would like to thank my advisor Cameron Gordon for his guidance and very helpful suggestions, as well as the mathematics department at the University of Texas, Austin for my support via the NSF grant DMS-1148490.

\bibliography{twistedalex}

\end{document}